\documentclass{amsart}
\usepackage{amssymb}
\usepackage{enumerate}
\usepackage{url}
\usepackage{amsmath}
\usepackage{graphicx}
\usepackage{amsfonts}
\usepackage{amssymb}
\usepackage{amsthm}
\usepackage{hyperref}
\usepackage{dsfont}

\newtheorem{theorem}{Theorem}[section]
\newtheorem{lemma}[theorem]{Lemma}

\newtheorem{proposition}[theorem]{Proposition}

\theoremstyle{definition}
\newtheorem{definition}[theorem]{Definition}
\newtheorem{example}[theorem]{Example}
\newtheorem{remark}[theorem]{Remark}

\newcommand\N{\mathbb{N}}
\newcommand\R{\mathbb R}

\newcommand\B{\mathcal{B}}

 \newcommand{\ep}[0]{\varepsilon}
 
 \newcommand{\al}[0]{\alpha}
 \newcommand{\be}[0]{\beta}

 \newcommand{\de}[0]{\delta}

 \newcommand{\la}[0]{\lambda}
 \newcommand{\si}[0]{\sigma}

 \newcommand{\Ph}[0]{\Phi}
 \newcommand{\Ps}[0]{\Psi}
 
\newcommand{\No}{\mathcal{N}}
\newcommand{\Po}{\mathcal{P}}
\newcommand{\El}{\mathcal{E}}
\newcommand{\X}{\underline{X}}
\newcommand{\M}{\mathcal{M}}

\newcommand\W{\mathcal{W}}\newcommand\F{\mathcal{F}}\newcommand\m{\mathfrak{m}}

\begin{document}
\title[Representation and Approximation of Positivity Preservers]
{Representation and Approximation of Positivity Preservers}
\author{Tim Netzer}
\address{Universit\"at Konstanz, Fachbereich Mathematik und Statistik, 78457 Konstanz, Germany}
\email{tim.netzer@gmx.de} 
\date{\today}
\begin{abstract}
We consider a closed set $S\subseteq\R^n$ and a linear operator
$$\Ph\colon\R[X_1,\ldots,X_n]\rightarrow \R[X_1,\ldots,X_n]$$ that
preserves nonnegative polynomials, in the following sense: if
$f\geq 0$ on $S$, then $\Ph(f)\geq 0$ on $S$ as well. We show that
each such operator is given by integration with respect to a
 measure taking nonnegative functions as its values. This can be seen as a generalization of
Haviland's Theorem, which concerns linear \textit{functionals} on
$\R[X_1,\ldots,X_n]$. For compact sets $S$ we use the result to
show that any nonnegativity preserving operator is a pointwise
limit of very simple nonnegativity preservers with finite
dimensional range.
\end{abstract}

\maketitle

\section{Introduction}
Linear operators that preserve \textit{hyperbolic} polynomials
have already been studied a hundred years ago by P\'{o}lya and
Schur \cite{ps}. A univariate real polynomial $p$ is called
hyperbolic, if all of its roots are real, and a linear map
$\Ph\colon\R[t]\rightarrow\R[t]$ is called a \textit{hyperbolicity
preserver}, if for any hyperbolic $p\in\R[t]$, the image $\Ph(p)$
is again hyperbolic. For example, simple differentiation $p\mapsto
\frac{\partial}{\partial t}p$ is a hyperbolicity preserver, which
follows from Rolle's Theorem.

In \cite{gs1}, the authors study ellipticity-, positivity- and
nonnegativity-preserving operators on polynomial algebras.
Although there is a huge amount of literature on positive
operators in Banach lattices (see for example \cite{ab}), results
specific to the polynomial case seem to be rare.

Let $\R[\underline{X}]=\R[X_1,\ldots,X_n]$ be the real polynomial
algebra in $n$ variables. Let
$$\No(\R^n):=\left\{p\in\R[\underline{X}]\mid p(x)\geq 0 \mbox{ for
all } x\in\R^n\right\}$$ denote the set of \textit{globally
nonnegative polynomials}, let
$$\Po(\R^n):=\left\{p\in\R[\X]\mid p(x)>0 \mbox{ for all }
x\in\R^n\right\}$$ be the set of \textit{globally positive
polynomials}, and let $$\El(\R^n)=\left\{p\in\R[\X]\mid p(x)\neq 0
\mbox{ for all } x\in\R^n\right\}$$ be the set of polynomials
without real zeros, also called \textit{elliptic polynomials}.

A linear map $\Ph\colon\R[\X]\rightarrow\R[\X]$ is called
\textit{nonnegativity}-, \textit{positivity}- or
\textit{ellipticity-preserving}, if
$\Ph\left(\No(\R^n)\right)\subseteq \No(\R^n)$,
$\Ph\left(\Po(\R^n)\right)\subseteq \Po(\R^n)$ or
$\Ph\left(\El(\R^n)\right)\subseteq \El(\R^n)$ holds,
respectively. One wants to characterize and describe these kinds
of operators as good as possible. It turns out that this question
is closely related to both real algebra  as well as to functional
analysis.

Fundamental work towards a characterization of all nonnegativity-,
positivity- and ellipticity-preservers is done in \cite{gs1,gs2},
and \cite{b} then contains a full characterization of such
operators, in terms of  differential operator representations. So
the problem can be considered as completed, where of course other
kinds of characterizations would still be of interest.

In this work we consider operators that preserve polynomials which
are positive, nonnegative or elliptic on a certain \textit{subset}
$S$ of $\R^n$. All the results in the spirit of \cite{b,gs1,gs2}
that one might expect turn out to be false in the general case.

We therefore choose a different approach to the problem. Our first
main result is an integral representation of general nonnegativity
preservers. Haviland's Theorem says that every linear functional
on $\R[\X]$ that maps $S$-nonnegative polynomials to nonnegative
reals is always given by integration on $S$. A similar statement
is true for $S$-nonnegativity preserving linear operators on
$\R[\X]$. The occuring measures can of course not be real-valued
in general; they take certain nonnegative functions as their
values instead. Our main integral representation result is Theorem
\ref{weak} below. It applies to any closed set $S$ and any
$S$-nonnegativity preserver.

In the case of a compact set $S$, the integral representation can
be strengthened (Theorem \ref{strong}). Compared to the standard
integral representation results for operators in $C(S)$ (as in
\cite{ds}), it is different in the sense that it does not assume
compactness or weak compactness of the operator; it uses its
nonnegativity instead. The occuring measures can be used to check
whether an $S$-nonnegativity preserver is of finite dimensional
range (Theorem \ref{fira}), compact (Theorem \ref{co}) or weakly
compact (Theorem \ref{weco}). The last two results are standard
from the representation theory of operators on $C(S)$, where the
first one is specific to nonnegative operators on $\R[\X]$.

In the last section we propose another possible solution to the
classification problem of nonnegativity preservers. The idea is to
first provide a class of standard operators that preserve
nonnegativity, and then check which operators can be approximated
by these operators in a suitable sense. In the case of a compact
set $S$, Theorem \ref{approx} below is such an approximation
result.

\section{Preliminaries and known results} Let $S\subseteq\R^n$ be
a set. Similar as before, let \begin{align*}
\No(S)&=\left\{p\in\R[\X]\mid p(x)\geq 0 \mbox{ for all } x\in
S\right\} \\ \Po(S)&=\left\{p\in\R[\X]\mid p(x) >0 \mbox{ for all
} x\in S\right\} \\ \El(S)&=\left\{p\in\R[\X]\mid p(x)\neq 0
\mbox{ for all } x\in S\right\}
\end{align*} denote the set of polynomials that
are nonnegative, positive and elliptic on $S$, respectively. We
clearly have $\Po(S)\subseteq\No(S)$ and $\pm\Po(S)\subseteq
\El(S)$. If $S$ is connected, then $\El(S)=\Po(S)\cup -\Po(S)$.

A linear map (also called operator)
$\Ph\colon\R[\X]\rightarrow\R[\X]$ is called an
\textit{$S$-nonnegativity preserver}, if
$\Ph\left(\No(S)\right)\subseteq\No(S);$ it is called an
\textit{$S$-positivity preserver} if
$\Ph\left(\Po(S)\right)\subseteq\Po(S)$, and an
\textit{$S$-ellipticity preserver} if
$\Ph\left(\El(S)\right)\subseteq\El(S)$ holds.

We say that $S$ is \textit{Zariski dense}, if it is not contained
in the zero set of a polynomial from $\R[\X]\setminus\{0\}$. In
that case, any two different polynomials define different
functions on $S$. We will assume the Zariski denseness of $S$ most
of the time.

The following result is a generalized combination of Lemma 2.1,
Theorem 2.3  and Theorem 2.5 from \cite{gs1}. We include a short
alternative proof.

\begin{proposition}\label{pre} (i) Each $S$-positivity preserver is an
$S$-nonnegativity preserver.

(ii) Let $\Ph\colon\R[\X]\rightarrow\R[\X]$ be an
$S$-nonnegativity preserver. Assume that $\Ph(1)(x)= 0$ for some
$x\in S$. Then $\Ph(p)(x)= 0$  for all $p\in\R[\X]$. In
particular, if $S$ is Zariski dense, then $\Ph(1)=0$ implies
$\Ph=0$.

(iii) If either $S$ is compact, or $n=1$ and $S\subseteq\R$ is
closed, then for any $S$-nonnegativity preserver $\Ph$ and any
$p\in\Po(S)$, the zero locus of $\Ph(p)$ in $S$ equals the zero
locus of $\Ph(1)$ in $S$. In particular, $\Ph$ is an
$S$-positivity preserver if and only if $\Ph(1)>0$ on $S$ in that
case.

 (iv) If $S$ is connected, then $\Ph$ is an $S$-ellipticity
preserver if and only if $\Ph$ or $-\Ph$ is an $S$-positivity
preserver.
\end{proposition}
\begin{proof}(i) is proven exactly as in \cite{gs1}, Theorem 2.1.
Namely, if $\Ph$ is $S$-positivity preserving and $p\geq 0$ on
$S$, then for all $\ep>0$, $$\Ph(p)+\ep \Ph(1)= \Ph(p+\ep)>0
\mbox{ on } S,$$ so $\Ph(p)\geq 0$ on $S$.

For (ii) let $p\in\R[\X]$ be arbitrary and note that
$$0\leq \Ph\left( (p+\la)^2 \right)(x) = \Ph(p^2)(x) +2\la \Ph(p)(x)$$ holds for any $\la\in\R$. So clearly
$\Ph(p)(x)=0$. Now if $\Ph(1)=0$, then $\Ph(p)(x)=0$ for all $p$
and all $x\in S$. The Zariski denseness of $S$ then implies
$\Ph=0$.

(iii) is again proven similar to \cite{gs1}, Theorem 2.5: If $p>0$
on $S$, then $p\geq \ep$ on $S$ for some suitable $\ep>0$ (since
$S$ is compact or $S\subseteq\R$ closed). So
$$\Ph(p)-\ep\Ph(1)=\Ph(p-\ep)\geq 0 \mbox{ on } S,$$ for any
$S$-nonnegativity preserver $\Ph$. So the zero locus of $\Ph(p)$
in $S$ is contained in the zero locus of $\Ph(1)$ in $S$. Equality
follows from (ii).

 (iv) Let $S$ be connected and
let $\Ph$ be $S$-ellipticity preserving. $\Ph(1)$ has no zeros in
$S$, so suppose $\Ph(1)>0$ on $S$ (otherwise replace $\Ph$ by
$-\Ph$). Now let $p\in\R[\X]$ be strictly positive on $S$. Then
for any $\la\in[0,1]$, the polynomial $\la p+(1-\la)$ does not
have zeros in $S$, so
$$\Ph(\la p + (1-\la))=\la \Ph(p)+ (1-\la) \Ph(1)$$ does not have
zeros in $S$, for any $\la\in[0,1]$. As $\Ph(1)>0$ on $S$, this is
clearly only possible if also $\Ph(p)>0$ on $S$. So $\Ph$ is
$S$-positivity preserving. The other direction follows immediately
from $\El(S)=\Po(S)\cup -\Po(S)$.
\end{proof}

So in view of Proposition \ref{pre}, we can restrict ourself to
examining $S$-nonnegativity preservers, at least in the case
$S\subseteq\R$ closed or $S$ compact.

Before we can describe the main results from \cite{b,gs1,gs2}, we
introduce certain classes of linear operators on $\R[\X]$.
Therefore let always $\N=\{0,1,2,\ldots\}.$ For any
$\al=(\al_1,\ldots,\al_n),\be=(\be_1,\ldots,\be_n)\in\N^n$, let
$\al!:=\al_1!\cdots\al_n!$, $|\al|:=\al_1+\cdots +\al_n$, write
$\al\preceq\be$ if $\al_i\leq\be_i$ for all $i$, and define
$\be-\al:=(\be_1-\al_1,\ldots,\be_n-\al_n)$. We also use the
notation $\X^{\al}$ for $X_1^{\al_1}\cdots X_n^{\al_n}.$

Here is a list of certain kinds of linear operators on $\R[\X]$:
\begin{example}\label{exam}
\begin{itemize}\label{list}\item[(1)] For $\al\in\N^n$, let $D^{\al}$ denote
the corresponding  differential operator, i.e. the linear operator
that sends $\X^{\be}$ to
$\frac{\be!}{(\be-\al)!}\cdot\X^{\be-\al}$ for $\al\preceq\be$,
and to $0$ otherwise.\item[(2)] Let $f\in\R[\X]$ be a fixed
polynomial, then the multiplication with $f$ is a linear operator,
denoted by $M_f$. Clearly, $M_f$ is $S$-nonnegativity preserving
if and only if $f\in\No(S)$, and $S$-positivity preserving if and
only if $f\in\Po(S)$. \item[(3)] Let $f_1,\ldots,f_n\in\R[\X]$ be
fixed polynomials, and consider the operator
$E_{\underline{f}}\colon p\mapsto p(f_1,\ldots,f_n)$. Such an
operator is even multiplicative, i.e. it is an $\R$-algebra
endomorphism of $\R[\X]$. Each algebra endomorphism is of that
form. To see whether $E_{\underline{f}}$ is $S$-nonnegativity or
$S$-positivity preserving, consider the corresponding polynomial
map
$$\underline{f}\colon\R^n\rightarrow\R^n; x\mapsto
(f_1(x),\ldots,f_n(x)).$$ One checks that $E_{\underline{f}}$ is
$S$-nonnegativity preserving if and only if
$\underline{f}\left(S\right)\subseteq \overline{S}$, and
$E_{\underline{f}}$ is $S$-positivity preserving if and only if
$\underline{f}\left(S\right)\subseteq S$. \item[(4)] A special
class of operators is the following: for polynomials
$f_1,\ldots,f_r$ and points $x_1,\ldots,x_r\in\R^n$ consider
$$\Ph_{\underline{f},\underline{x}}\colon p\mapsto f_1\cdot
p(x_1)+\cdots+f_r\cdot p(x_r).$$ This operator has a finite
dimensional range; in case that the $x_i$ are pairwise disjoint,
its range equals the subspace of $\R[\X]$ spanned by
$f_1,\ldots,f_r$. If all $f_i$ are nonnegative on $S$ and all
$x_i\in S$, the operator is clearly $S$-nonnegativity preserving.
These operators will be used in the approximation result below.
\item[(5)] A generalization of (4) is the following. Let
$L_1,\ldots,L_r$ be linear functionals on $\R[\X]$ and
$f_1,\ldots,f_r\in\R[\X]$. Then
$$\Ph_{\underline{f},\underline{L}}\colon p\mapsto \sum_{i=1}^r f_i\cdot L_i(p)$$ is a linear operator on
$\R[\X]$. It also has a finite dimensional range; in case that all
the $L_i$ are linearly independent, it is the subspace spanned by
the $f_i$. Each linear operator with finite dimensional range is
of this form (see for example in the proof of Theorem \ref{fira}).
If all $f_i\geq 0$ on $S$ and all $L_i$ map $S$-nonnegative
polynomials to $[0,\infty)$, this operator is $S$-nonnegativity
preserving; however, it should be noted that this is not necessary
(see open problem (4) below). \item[(6)] Any two linear operators
$\Ph,\Ps$ on $\R[\X]$ can be summed up and composed, i.e. one can
consider $\Ph+\Ps\colon p\mapsto \Ph(p)+\Ps(p)$ and
$\Ph\circ\Ps\colon p\mapsto \Ph\left(\Ps(p)\right).$ The sum and
composition of two $S$-nonnegativity or $S$-positivity preservers
is again an $S$-nonnegativity or $S$-positivity preserver,
respectively.
 \item[(7)] Let $\left(\Ph_i\right)_{i\in
I}$ be a familiy of linear operators on $\R[\X]$, such that for
any polynomial $p\in\R[\X]$ only finitely many of the values
$\Ph_i(p)$ are not zero. Then $$\sum_{i\in I}\Ph_i\colon p\mapsto
\sum_{i\in I} \Ph_i(p)$$ is a well defined linear operator. For
example, with a (multi-)sequence
$\left(q_{\al}\right)_{\al\in\N^n}$ of polynomials we can define
the operator $$\sum_{\al\in\N^n}q_\al D^{\al}\colon p\mapsto
\sum_{\al}q_{\al}\cdot D^{\al}(p).$$ Such an operator is called a
\textit{differential operator with polynomial coefficients}. If
only finitely many of the polynomials $q_{\al}$ are non-zero, then
the operator is called \textit{of finite order}; otherwise it is
called \textit{of infinite order}. If the $q_{\al}$ are all real
numbers, then the operator is called a \textit{differential
operator with constant coefficients}.
\end{itemize}\end{example}

The following fact is folklore, we include a short proof for the
sake of completeness:

\begin{lemma}Every linear operator on $\R[\X]$ is a
differential operator with polynomial coefficients. The
corresponding multisequence $\left(q_{\al}\right)_{\al\in\N^n}$ of
polynomial coefficients is unique.
\end{lemma}
\begin{proof} Let $\Ph$ be a linear operator on $\R[\X]$ and set
$p_{\be}:= \Ph(\X^{\be})$ for any $\be\in\N^n$. We have to find
polynomials $q_{\al}$ such that $p_{\be}=\sum_{\al\preceq\be}
\frac{\be!}{(\be-\al)!}q_{\al}\X^{\be-\al}$ for all $\be$. This is
done by induction on $|\al|$. We first choose $q_0=p_0$. Then for
any $\be\neq 0$, from $$p_{\be}=\sum_{\al\preceq\be}
\frac{\be!}{(\be-\al)!}q_{\al}\X^{\be-\al}=
q_{\be}\be!+\sum_{\al\preceq\be, |\al|<|\be|}
\frac{\be!}{(\be-\al)!}q_{\al}\X^{\be-\al}$$ we deduce $$q_{\be}=
\frac{1}{\be!}\left( p_{\be} - \sum_{\al\preceq\be, |\al|<|\be|}
\frac{\be!}{(\be-\al)!}q_{\al}\X^{\be-\al}\right).$$ The
differential operator with coefficient sequence
$\left(q_{\al}\right)_{\al}$, defined inductively by the above
rule, then coincides with $\Ph$, and there is clearly no other
possible sequence of coefficients for $\Ph$.
\end{proof}

\begin{example}(i) The multiplication
operator $M_f$ defined above is already in the form of a
differential operator, namely $M_f=f \cdot D^0.$

(ii) Consider the algebra endomorphism $E_{\underline{f}}$ as
defined above. For any polynomial $p\in\R[\X]$ we have
\begin{align}p(f_1,\ldots,f_n)=\sum_{\al\in\N^n}\frac{1}{\al!}\left(f_1-X_1\right)^{\al_1}\cdots\left(f_n-X_n\right)^{\al_n}\cdot
D^{\al}(p).\end{align} This Taylor-formula is easily verified for
monomials $p=\X^{\be},$ and thus holds in general. Recall that the
above sum is always finite, there is no convergence problem. So
the sequence $\left(q_{\al}\right)_{\al\in\N^n}$ defined by
$$q_{\al}:=\frac{1}{\al!}\left(f_1-X_1\right)^{\al_1}\cdots\left(f_n-X_n\right)^{\al_n}$$
is the coefficient sequence for $E_{\underline{f}}$ in its
representation as a differential operator.
\end{example}

\begin{definition}A (multi-)sequence
$\left(r_{\al}\right)_{\al\in\N^n}$ of real numbers is called a
\textit{moment sequence}, if there is a nonnegative Borel measure
$\mu$ on $\R^n$ such that $$r_{\al}=\int_{\R^n} \X^{\al}d\mu$$
holds for all $\al\in\N^n$.
\end{definition}

The following Theorem sums up some of the most important results
from \cite{b,gs1,gs2}. It can be seen as a complete
characterization of $\R^n$-nonnegativity preserving operators:

\begin{theorem}\label{globmain}(i) A differential operator $\Ph=\sum_{\al\in\N^n}
r_{\al}D^{\al}$ with constant coefficients is $\R^n$-nonnegativity
preserving if and only if the sequence
$\left(\al!r_{\al}\right)_{\al\in\N^n}$ is a moment sequence.

(ii) A differential operator $\Ph=\sum_{\al\in\N^n}q_{\al}D^{\al}$
with polynomial coefficients is $\R^n$-nonnegativity preserving if
and only if for all $a\in\R^n$ the operator
$$\Ph_a:=\sum_{\al\in\N^n}q_{\al}(a)D^{\al}$$ is
$\R^n$-nonnegativity preserving (and (i) applies to each $\Ph_a$).

(iii) A differential operator of finite order (with constant or
polynomial coefficients) can only be $\R^n$-nonnegativity
preserving if it is of order $0$, i.e. if it is of the form $M_f$
for some $f\in\R[\X]$.
\end{theorem}

Part (i) is \cite{b} Theorem 3.1, a special case is \cite{gs1},
Theorem 3.4 and also \cite{gs2}, Theorems A and B. Part (ii) is
again \cite{b}, Theorem 3.1. Part (iii) follows from (i) and (ii),
and was first proven in a constructive way in \cite{gs1}, Section
4.

So the problem of characterizing $\R^n$-nonnegativity preservers
boils down to characterizing moment sequences. There is a huge
amount of literature on the moment problem, we only refer to
\cite{a,h,ha,sm2} and the references therein. An important result
is for example Hamburger's Theorem for the one dimensional case:
\begin{theorem} A sequence ${\bf r}=\left(r_i\right)_{i\in\N}$ is a moment
sequence if and only if for all $m\in\N$ the matrix
$\mathcal{H}({\bf r})_m:=\left( r_{i+j}\right)_{i,j=0}^m$ is
positive semidefinite.
\end{theorem} Note however that in the case
of dimension $n \geq 2$, there is not such an easy
characterization of moment sequences.

We will also need  Haviland's Theorem in the following (\cite{ha},
see also \cite{m} for a proof):
\begin{theorem}[Haviland]\label{havi} Let $S\subseteq\R^n$ be a
closed set and let $L\colon\R[\X]\rightarrow\R$ be a linear
functional. Then there is a nonnegative Borel measure $\mu$ on $S$
such that $$L(p)=\int_S p d\mu \mbox{ for all } p\in\R[\X]$$ if
and only if $L(p)\geq 0$ whenever $p\geq 0$ on $S$.
\end{theorem}

Hamburger's Theorem follows for example from  Haviland's Theorem
and the fact that every nonnegative univariate polynomial is a sum
of squares of polynomials.

\section{Examples in the generale case}In this section we consider possible  generalizations of
Theorem \ref{globmain} to the case of arbitrary sets
$S\subseteq\R^n$. So let $S\subseteq\R^n$ be arbitrary. An
\textit{$S$-moment sequence} is a sequence
$\left(r_{\al}\right)_{\al\in\N^n}$ of real numbers, such that
there exists a nonnegative Borel measure $\mu$ on $S$ with
$$r_{\al}=\int_S \X^{\al}d\mu\quad  \mbox{ for all }\al\in\N^n.$$
The value $\int_S \X^{\al} d\mu$ is called the
\textit{($S$-)moment of $\mu$ of order $\al$}.

 The natural guess how to generalize Theorem
\ref{globmain} would be to simply replace $\R^n$ by $S$ each time.
But only some parts remain true.

\begin{lemma}\label{not}If $0 \in S$ and a differential operator
$\Ph=\sum_{\al}r_{\al}D^{\al}$ with constant coefficients is
$S$-nonnegativity preserving, then the sequence
$\left(\al!r_{\al}\right)_{\al\in\N^n}$ is an
$\overline{S}$-moment sequence.\end{lemma}
\begin{proof} The proof is the same as in \cite{b}, Theorem 3.1.
We consider the linear functional $p\mapsto \Ph(p)(0)$, which maps
$\No(S)$ to $[0,\infty)$. So by Haviland's Theorem, there exists a
measure $\mu$ on $\overline{S}$ such that
$$\Ph(p)(0)=\int_{\overline{S}}p d\mu$$ for all $p\in\R[X]$. In
particular $$\int \X^{\be}
d\mu=\Ph(\X^{\be})(0)=\sum_{\al\preceq\be} r_{\al}
\frac{\be!}{(\be-\al)!} \X^{\be-\al}(0)=\be!r_{\be}$$ for all
$\be\in\N^n$.
\end{proof}

Of course, one can always ensure $0\in S$ by switching to
$$\Ph'\colon p\mapsto \Ph\left(p(\X+a)\right)(\X-a)$$ for some
suitable $a\in\R^n$, if $S\neq \emptyset$. However, without the
assumption $0\in S$, Lemma \ref{not} fails:

\begin{example} Let $n=1$ and $S:=[2,\infty)\subseteq\R$. The
operator $$E_{X+1}\colon p\mapsto p(X+1)=\sum_{i=0}^{\infty}
\frac{1}{i!}D^{i}(p)$$ is even $S$-positivity preserving. But
there is no Borel measure $\mu$ on $[2,\infty)$ such that
$$\int_2^{\infty}X^id\mu = 1 \mbox{ for all } i\in \N.$$ Indeed from
$1=\int_2^{\infty}X^0d\mu=\mu\left([2,\infty)\right)$ we would
obtain $$1=\int_2^{\infty}X d\mu\geq 2\cdot
\mu\left([2,\infty)\right)=2,$$ a contradiction.
\end{example}

The converse of Lemma \ref{not} fails even if $0\in S$:

\begin{example} Let $S=[-1,0]\subseteq\R$ and let $\mu$ be the
Lebesque-measure restricted to $S$. Define
$$r_i:=\frac{1}{i!}\int_{-1}^0 X^i d\mu$$ for all $i\in \N$. Then
for the linear operator $\Ph$ defined by
$\Ph:=\sum_{i=0}^{\infty}r_iD^i$ we have
$$\Ph(X+1)(-1)=\sum_{i=0}^{\infty}r_i D^i(X+1)(-1)=r_1=\int_{-1}^0X
d\mu=-\frac12<0.$$ As $X+1\in\No(S)$, $\Ph$ is not
$S$-nonnegativity preserving.
\end{example}

Now we turn to possible generalizations of (ii) in Theorem
\ref{globmain}.

\begin{lemma}Let $S\subseteq\R^n$ and let $\Ph=\sum_{\al\in\N^n}q_{\al}D^{\al}$ be a differential operator with polynomial coefficients. Suppose $\Ph_a=\sum_{\al\in\N^n}q_{\al}(a)D^{\al}$
 is an $S$-nonnegativity preserver for all $a\in S$. Then $\Ph$ is
 an $S$-nonnegativity preserver.
\end{lemma}
\begin{proof} As in the proof of \cite{b}, Theorem 3.1, the result
is clear from $\Ph(p)(a)=\Ph_a(p)(a)$.
\end{proof}

Again, the converse of this Lemma fails in general:

\begin{example}Let $S:=[-1,1]\subseteq\R$ and let $$\Ph=E_{\frac{X}{2}}\colon
p\mapsto
p\left(\frac{X}{2}\right)=\sum_{i=0}^{\infty}\frac{1}{i!}\left(-\frac{X}{2}\right)^i
D^i(p),$$ which is an $S$-positivity preserver. We have $1\in S$
and
$$\Ph_1=\sum_{i=0}^{\infty}\frac{1}{i!} \left(-\frac{1}{2}\right)^iD^i,$$
so $\Ph_1(p)=p\left(X-\frac{1}{2}\right)$ for all $p$, again using
the Taylor-formula. The polynomial $p=X+1$ belongs to $\No(S)$,
but $\Ph_1(p)=X+\frac{1}{2}$ does not. So $\Ph_1$ is not
$S$-nonnegativity preserving.
\end{example}

So in the case of arbitrary sets $S\subseteq\R^n$, one has to look
for different kinds of characterizations of $S$-nonnegativity
preservers.

\section{The adjoint map} We introduce the \textit{adjoint map} to a
nonnegativity preserver.  This adjoint map, defined via Haviland's
Theorem, was implicitly already used in \cite{b,gs1,gs2}.

Let $S\subseteq\R^n$ be closed and
$\Ph\colon\R[\X]\rightarrow\R[\X]$ be an $S$-nonnegativity
preserver. For every nonnegative Borel measure $\mu$ on $S$ with
all finite moments consider the map
$$L_{\mu}\colon\R[\X]\rightarrow \R; \ p\mapsto
\int_S\Ph(p)d\mu.$$ $L_{\mu}$ is a linear functional and satisfies
the assumption from Haviland's Theorem, since $\Ph$ is
$S$-nonnegativity preserving. So there is another nonnegative
Borel measure $\nu$ on $S$ with all finite moments, such that
$$\int_S \Ph(p)d\mu=\int_S p d\nu \quad \forall p\in\R[\X].$$

Let $\M_+(S)$ denote the set of all nonnegative Borel measures on
$S$ with all finite moments. $S$ can be considered as a subset of
$\M_+(S)$, by identifying each point $x\in S$ with the Dirac
measure $\de_x$ centered at $x$.

\begin{definition} A map $T\colon\M_+(S)\rightarrow\M_+(S)$ is called
an \textit{adjoint map} to $\Ph$, if $$\int_S\Ph(p)d\mu=\int_S pd
T(\mu)$$ holds for all $p\in\R[\X]$ and all $\mu\in\M_+(S)$.
\end{definition}

\begin{proposition}\label{functor}Let $S\subseteq\R^n$ be closed. Then

(i) Every $S$-nonnegativity preserver $\Ph$ has an adjoint map.

(ii) If $T$ and $U$ are adjoint maps to $\Ph$ and $\Ps$
respectively, then $T+U$ is an adjoint map to $\Ph+\Ps$ and
$U\circ T$ is an adjoint map to $\Ph\circ\Ps.$

(iii) If $S$ is Zariski dense and $T$ is adjoint to both $\Ph$ and
$\Ps$, then $\Ph=\Ps$.

(iv) If $f\geq 0$ on $S$, then an adjoint map for the
multiplication operator $M_f$ is the map $$T_f\colon\mu\mapsto
f\cdot \mu,$$ where $\left(f\cdot\mu\right)(A):=\int_A fd\mu$ for
Borel sets $A\subseteq\R^n$.

(v) Let $f_1,\ldots,f_n\in\R[\X]$ be such that $E_{\underline{f}}$
is an $S$-nonnegativity preserver. Then the map
$$T_{\underline{f}}\colon\mu\mapsto \mu\circ \underline{f}^{-1},$$ (where
$\left(\mu\circ
\underline{f}^{-1}\right)(A):=\mu\left(\underline{f}^{-1}(A)\right)$
for Borel sets $A$) is an adjoint map to $E_{\underline{f}}.$
\end{proposition}
\begin{proof}(i) is clear from Haviland's Theorem, as explained
above, (ii), (iv) and (v) are standard results from intergration
theory.

Towards (iii) let $\de_x$ denote the Dirac measure centered at
$x$, for every point $x\in S$. We have $$\Ph(p)(x)=\int_S\Ph(p)
d\de_x =\int_S pd T\left(\de_x\right)=\int_S
\Ps(p)d\de_x=\Ps(p)(x)$$ for every $p\in\R[\X]$ and $x\in S$. So
$\Ph=\Ps$, by the Zariski denseness of $S$.
\end{proof}

Note that the adjoint map is not unique in general. The question
whether the measure obtained in Haviland's Theorem is unique is
also known as the "determinate moment problem" (see for example
\cite{pusc,pusm,puva}).

\begin{remark} If $S$ is compact, then the measures in Haviland's
Theorem are unique, as for example pointed out in \cite{m},
Section 3.3. So the adjoint map for an $S$-nonnegativity preserver
$\Ph$ is also unique in that case, and we denote it by $\Ph^{*}$.
In view of Proposition \ref{functor} we have
$\left(\Ph+\Ps\right)^*=\Ph^*+\Ps^*$,
$\left(\Ph\circ\Ps\right)^*=\Ps^*\circ\Ph^*$ and if $S$ is Zariski
dense, one has
$$\Ph^*=\Ps^*\Longleftrightarrow \Ph=\Ps.$$
For $x\in S$ we also write $\mu_x$ instead of
$\Ph^{*}\left(\de_x\right)$, if no confusion can arise (i.e. if
only one operator is considered). With this notation we have
\begin{align}\label{eq}\Ph(p)(x)=\int_S p d\mu_x\end{align} for all
$p\in\R[\X],x\in S$.
\end{remark}

\section{Integral representations of $S$-nonnegativity preservers}

In this section we construct a general integral representation of
$S$-nonnegativity preservers. It can be seen as a generalization
of Haviland's Theorem, and does not assume compactness of $S$ or
any continuity of operators.

 Let $S\subseteq\R^n$ be closed and
$\Ph$ an $S$-nonnegativity preserver. Suppose $T$ is an adjoint
map to $\Ph$. Then for each Borel set $A\subseteq S$ we can
consider the map
\begin{align*}T_A\colon & S\rightarrow \R\\ &x\mapsto
T(\de_x)(A).\end{align*}Write $\mathcal{B}(S)$ for the
$\si$-algebra of Borel subsets of $S$. For any $A\in\B(S)$ and any
$x\in S$ we have $$0\leq T_A(x)\leq T_S(x)=\int_S 1d
T(\de_x)=\Ph(1)(x).$$ So all $T_A$ are nonnegative functions on
$S$ that are bounded by $\Ph(1)$, pointwise on $S$.

\begin{example}\label{example}(i) If $\Ph$ is the identity operator, then the identity map $T$
is an adjoint map to $\Ph$. We have $T_A=\mathds{1}_A$, the
characteristic function of $A.$  More general, consider a
multiplication operator $M_f$ with adjoint map $T_f$ as defined in
Proposition \ref{functor} (iv). Then ${T_f}_A=f\cdot\mathds{1}_A.$
 We see that the maps $T_A$ can not be expected to be continuous in
general.

(ii) For an $S$-nonnegativity preserving algebra homomorphism
$E_{\underline{f}}$ with adjoint map $T_{\underline{f}}$ as above
we have
$${T_{\underline{f}}}_A=\mathds{1}_{\underline{f}^{-1}(A)}.$$
 (iii) For finitely many polynomials $f_1,\ldots,f_r\geq 0$ on
$S$ and points $x_1,\ldots,x_r\in S$ consider the
$S$-nonnegativity preserver
$$\Ph_{\overline{f},\overline{x}}\colon p\mapsto f_1\cdot p(x_1)
+\cdots + f_r\cdot p(x_r).$$ An adjoint map is
$$T\colon \mu\mapsto \sum_{i=1}^r \left(\int f_id\mu\right) \cdot
\de_{x_i},$$so we have $$T_A=\sum_{\{i\mid x_i\in A\}}f_i,$$ a
polynomial map that is in particular continuous.
\end{example}

We always  have $T_{\emptyset}=0$ and for a countable family
$\left(A_i\right)_{i\in\N}$ of pairwise disjoint elements from
$\B(S)$ $$T_{\bigcup_i A_i}=\sum_i T_{A_i}.$$ Write $\F(S)$ for
the vector space of real valued and polynomially bounded functions
on $S$.
So the mapping \begin{align*}\m_T\colon \B(S)&\rightarrow \F(S)\\
A&\mapsto T_A\end{align*} is a measure  taking nonnegative
functions in $\F(S)$ as its values.

For pairwise disjoint sets $A_1,\ldots,A_t\in\B(S)$ and a real
valued step function $s=\sum_{i=1}^t r_i\cdot\mathds{1}_{A_i}$ we
define
$$\int_S sd\m_T:=\sum_{i=1}^t r_i\cdot\m_T(A_i)=\sum_{i=1}^r r_i\cdot T_{A_i},$$ which is a well
defined polynomially bounded real valued function on $S$. For
$x\in S$ we have $$\left(\int_S sd\m_T\right)(x)=\sum_i r_i\cdot
T_{A_i}(x)=\sum_i r_i\cdot T(\de_x)(A_i)=\int sd T(\de_x).$$

The following is our first main result. It is a general integral
representation of $S$-nonnegativity preservers and can be seen as
an generalization of Haviland's Theorem to operators:

\begin{theorem}\label{weak} Let $S\subseteq\R^n$ be closed and $\Ph\colon \R[\X]\rightarrow \R[\X]$ a
linear map. Then the following are equivalent:
\begin{itemize}\item[(i)] $\Ph$ is $S$-nonnegativity preserving
\item[(ii)] There is a measure $\m\colon\B(S)\rightarrow\F(S),$
taking only nonnegative functions as its values, with
$$\Ph(p)=\int_S pd\m \mbox{ for all } p\in\R[\X],$$ in the following sense:  For every sequence of
real valued step functions $(s_j)_j$ such that $|s_j|\leq q$ on
$S$ for a polynomial $q$ and all $j$, if $(s_j)_j$ converges
pointwise on $S$ to $p$, then the sequence $\left(\int_S s_j
d\m\right)_j$ converges pointwise on $S$ to $\Ph(p)$.
\end{itemize}
\end{theorem}
\begin{proof}(i)$\Rightarrow$(ii): Let $T$ be an adjoint map to
$\Ph$ and let $\m:=\m_T$ be as described above. Let $p\in\R[\X]$
and let $(s_j)_j$ be a sequence of step functions converging
pointwise on $S$ to $p$, with $|s_j|\leq q$ on $S$ for some
polynomial $q$. For any $x\in S$ we have
\begin{align*}\left|\left(\int s_j d\m\right)(x)-\Ph(p)(x)\right|&=
\left| \int_S s_j d T(\de_x) -\int_S pd T(\de_x)\right|\\&\leq
\int |s_j-p|d T(\de_x),\end{align*} which converges to zero for
$j\to\infty$, by the Theorem of Majorized Convergence, since
$\int_S qdT(\de_x)=\Ph(q)(x)<\infty$. (ii)$\Rightarrow$(i) is
clear from the fact that for every nonnegative polynomial $p$
there exists such an approximating sequence of step functions that
are nonnegative themselves.
\end{proof}

\section{Compact sets} During this whole section let
$S\subseteq\R^n$ be \textit{compact} and Zariski dense, and let
$\Ph$ be an $S$-nonnegativity preserver. Let $\Ph^*$ denote the
(unique) adjoint map to $\Ph$, as defined above. Again write
$\mu_x$ instead of $\Ph^*(\de_x)$, for $x\in S$. We show that the
integral representation from the last section holds in a stronger
sense.

Clearly,
$$||p||_{\infty}:=\max_{x\in S} |p(x)|$$ defines a norm on
$\R[\X]$. Any nonnegativity preserver is continuous with respect
to that norm:

\begin{lemma} Let $\Ph\colon\R[\X]\rightarrow\R[\X]$ be an
$S$-nonnegativity preserver. Then $\Ph$ is a continuous operator
with operator norm $||\Ph||= ||\Ph(1)||_{\infty}.$
\end{lemma}
\begin{proof} For $x\in S$ and $p\in\R[\X]$ we have \begin{align*}|\Ph(p)(x)|&=\left|\int_S p
d\mu_x\right|\\&\leq\int_S|p|d\mu_x \\&\leq
||p||_{\infty}\cdot\mu_x(S) \\&=||p||_{\infty}\cdot\int_S
1d\mu_x\\&= ||p||_{\infty}\cdot \Ph(1)(x) \\&\leq ||p||_{\infty}
\cdot ||\Ph(1)||_{\infty}.\end{align*} So $||\Ph(p)||_{\infty}\leq
||p||_{\infty}\cdot||\Ph(1)||_{\infty}$, thus $\Ph$ is continuous
with $||\Ph||\leq ||\Ph(1)||_{\infty}.$ Equality follows from
$||\Ph(1)||_{\infty}=||1||_{\infty}\cdot ||\Ph(1)||_{\infty}.$
\end{proof} As $\R[\X]$ embeds densely into $C(S)$, the Banach space of continuous
real valued functions on $S$, each $S$-nonnegativity preserver
$\Ph$ has a unique continuous extension $\widetilde{\Ph}$ to
$C(S)$. In the case of a compact set, the vector measures $\m$
defined above have nicer properties:

\begin{lemma}For any Borel set $A\subseteq S$, the mapping \begin{align*}\Ph^*_A\colon S&\rightarrow
[0,\infty)\\ x&\mapsto \mu_x(A)\end{align*} is Borel measurable
and we have $||\Ph^*_A||_{\infty}\leq ||\Ph(1)||_{\infty}.$
\end{lemma}

\begin{proof} $||\Ph^*_A||_{\infty}\leq ||\Ph(1)||_{\infty}$ is
clear for any $A$ by what we have shown above. Now let first
$A\subseteq S$ be closed. By Urysohn's Lemma and the
Stone-Weierstra\ss{} approximation, choose a sequence of
polynomials $(p_j)_j$ with $|p_j|\leq 2$ on $S$, that converge
pointwise on $S$ to $\mathds{1}_A$, the characteristic function of
$A$. We have $\Ph(p_j)(x)=\int p_j d\mu_x\stackrel{j}{\rightarrow}
\mu_x(A)=\Ph^*_A(x)$ for all $x\in S$, by the Theorem of Majorized
Convergence. So $\Ph^*_A$ is the pointwise limit of the polynomial
functions $\Ph(p_j)$, and so clearly measurable.

For any Borel set $A$ we have $$\Ph^*_{S\setminus A}=\Ph(1)-
\Ph^*_A$$ and for Borel sets $A_1\subseteq A_2 \subseteq \ldots$
$$\Ph^*_{\bigcup_i A_i}=\lim_i ´\Ph^*_{A_i}\quad  \mbox{ pointwise on } S,$$ so the general result follows by
transfinite induction.
\end{proof}

So the measure $\m_{\Ph^*}$ takes its values in a bounded subset
of  $\mathbb{B}(S)$, the Banach space of bounded measurable real
valued functions on $S$, equipped with the sup-norm. A whole
theory of integration with respect to a measure with values in an
arbitrary Banach space is for example developed in \cite{ds},
IV.10. In our case we obtain the following strong integral
representation of $\Ph$:

\begin{theorem}\label{strong} Let $S\subseteq\R^n$ be compact and Zariski dense. Let $\Ph\colon \R[\X]\rightarrow \R[\X]$ be
an $S$-nonnegativity preserver. Then the measure $\m:=\m_{\Ph^*}$
takes only nonnegative measurable functions as its values, and we
have
$$\Ph(p)=\int_S pd\m \mbox{ for all } p\in\R[\X],$$ in the sense of Theorem \ref{weak} (ii). Furthermore, if a sequence of
real valued step functions $(s_j)_j$ converges uniformly  on $S$
to $p$ (and such sequences exist for every $p$), then the sequence
$\int_S s_j d\m$ converges uniformly on $S$ to $\Ph(p)$.
\end{theorem}
\begin{proof} Clear from the above results and the proof of Theorem \ref{weak}.
\end{proof}

Note that the result does only assume that $\Ph$ is nonnegativity
preserving, in contrast to standard representation results as in
\cite{ds}, VI.7, where the operators have to be \textit{weakly
compact} (see there for notions as compact or weakly compact
operator). We can further investigate nonnegativity preservers and
their vector measures $\m_{\Ph^*}:$

\begin{theorem}\label{fira}Let $S\subseteq\R^n$ be compact and Zariski dense. Let $\Ph\colon \R[\X]\rightarrow \R[\X]$ be
an $S$-nonnegativity preserver. Then the following are equivalent:

\begin{itemize} \item[(i)] There is some $d\in\N$ such that all $\Ph^*_A$ are polynomials of degree $\leq d$. \item[(ii)] $\Ph$ has finite dimensional
range. \item[(iii)] $\widetilde{\Ph}$ has finite dimensional
range.
\end{itemize}
\end{theorem}
\begin{proof}

(ii)$\Leftrightarrow$(iii) is clear. (i)$\Rightarrow$(ii) is clear
from Theorem \ref{strong}. For (ii)$\Rightarrow$(i) let
$q_1,\ldots,q_r\in\R[\X]$ be such that
$f_1:=\Ph(q_1),\ldots,f_r:=\Ph(q_r)$ form a basis of
$\Ph(\R[\X])$. Using the fact that each polynomial is a difference
of two squares of polynomials, we can assume that all $q_i$ and
therefore all $f_i$ are nonnegative on $S$.

For each $p\in\R[\X]$ exist uniquely determined real numbers
$L_1(p),\ldots,L_r(p)$ such that $$\Ph(p)=\sum_{i=1}^r
L_i(p)f_i.$$ The mappings $L_i\colon\R[\X]\rightarrow \R$ are
linear functionals. In other words, we have
$\Ph=\Ph_{\underline{f},\underline{L}}$ (see Example \ref{exam}
(5)).

By the Riesz-Representation Theorem (\cite{ds} IV.6.3), there are
finite signed Borel measures $\nu_i$ on $S$ such that
$L_i(p)=\int_S pd\nu_i$ for all $p$ and $i$. By the
determinateness of the Moment Problem for compact sets, we have
$$\mu_x=\sum_{i=1}^r f_i(x)\cdot\nu_i\mbox{ for all } x\in S.$$ So
for each Borel set $A\subseteq S$ we have $$\Ph^*_A=\sum_{i=1}^r
\nu_i(A)\cdot f_i,$$ which proves the claim.
\end{proof}

The following two results are standard results from the theory of
compact and weakly compact operators:

\begin{theorem}\label{co}Let $S\subseteq\R^n$ be compact and Zariski dense. Let $\Ph\colon \R[\X]\rightarrow \R[\X]$ be
an $S$-nonnegativity preserver. Then the following are equivalent:

\begin{itemize}\item[(i)] The familiy $\left(\Ph^*_A\right)_{A\in\B(S)}$
is equicontinuous. \item[(ii)] $\widetilde{\Ph}$ is a compact
operator on $C(S)$.
\end{itemize}
\end{theorem}
\begin{proof} Use \cite{ds}, VI.7.7 Theorem 7 and the fact that a
subset of $C(S)$ is contained in a compact set if and only if it
is equicontinuous and pointwise bounded (this is the Arzela-Ascoli
Theorem).
\end{proof}

\begin{theorem}\label{weco}Let $S\subseteq\R^n$ be compact and Zariski dense. Let $\Ph\colon \R[\X]\rightarrow \R[\X]$ be
an $S$-nonnegativity preserver. Then the following are equivalent:
\begin{itemize} \item[(i)] All the functions $\Ph^*_A$ are
continuous. \item[(ii)] $\widetilde{\Ph}$ is a weakly compact
operator on $C(S)$.
\end{itemize}
\end{theorem}

\begin{proof} This is \cite{ds}, VI.7.3 Theorem 3.
\end{proof}

\begin{remark}In case that there is a finite positive Borel measure $\mu$
on $S$ such that
$$||\Ph(p)||_{\infty}\leq \int |p|d\mu$$ holds for all
$p\in\R[\X],$ it follows from \cite{bb} Theorem 8 that
$\widetilde{\Ph}$ is weakly compact. So all the functions
$\Ph^*_A$ are continuous in that case.\end{remark}

\section{Approximation of nonnegativity preservers}

In this section we show that all $S$-nonnegativity preservers can
be approximated by very simple ones, at least in the case of a
compact set $S$. The setup is the following. Let $\W$ denote the
set of all operators of the form
$$\Ph_{\underline{f},\underline{x}}\colon p\mapsto \sum_{i=1}^r
p(x_i)\cdot f_i,$$ with $r\in\N$, $x_1,\ldots,x_r\in S$ and
polynomials $f_1,\ldots,f_r$ which are nonnegative on $S$; see
also Example \ref{exam} (4). $\W$ is a convex cone contained in
the cone of all $S$-nonnegativity preservers. All elements from
$\W$ have a finite dimensional range. Our goal is to prove that
each $S$-nonnegativity preserver can be approximated by a sequence
of elements from $\W$, pointwise on $\R[\X]$ with respect to $||\
||_{\infty}$. This convergence is also often called
\textit{convergence in the strong operator topology}.

Therefore let $S$ be compact, $\Ph$ an $S$-nonnegativity preserver
and $\m:=\m_{\Ph^*}$ the function valued measure constructed from
$\Ph$ as above. Integration of bounded measurable functions with
respect to $\m$ is defined in \cite{ds}, IV.10. We have already
seen that for a real valued step function $s$ on $S$ and $x\in S$
we have
$$\left(\int_S sd\m\right)(x)=\int_S s d\mu_x.$$ So for any
bounded measurable function $h$ on $S$ the same formula remains
true. We will need it later in the proof of the main approximation
theorem.

The semi-variation $||\m||$ of $\m$ is defined as
$$||\m||(A):=\sup ||\sum_{i=1}^r \al_i\m(A_i)||_{\infty},$$ where
the supremum ranges over all finite collections of scalars with
$|\al_i|\leq 1$ and all partitions of $A$ into a finite number of
disjoint Borel sets $A_i$. $$0\leq ||\m(A)||_{\infty}\leq
||\m||(A)<\infty$$ holds for all Borel sets $A$. However, $||\m||$
can not be expected to be a measure in general, i.e. it is usually
not additive. But there always exists a finite positive measure
$\la$ on $S$ such that $\la(A)\leq ||\m||(A)$  and
$\la(A)=0\Leftrightarrow ||\m||(A)=0$ for all $A$ in $\B(S)$
(\cite{ds}, IV.10.5 Lemma 5). A Borel set $A$ is called an
\textit{$\m$-null set} if $||\m||(A)=0$ holds, or if $A$ is a
$\la$-null set, equivalently. The usual Theorem of Majorized
Convergence is true for integration with respect to $\m$
(\cite{ds} IV 10.10 Theorem 10).

The following is the announced approximation result:
\begin{theorem}\label{approx} Let $S\subseteq\R^n$ be compact and Zariski dense. Let $\Ph$ be an
$S$-nonnegativity preserver. Then $\Ph$ can be approximated by a
sequence of operators from $\W$, with respect to the strong
operator topology.
\end{theorem}
\begin{proof} First assume $A_1,\ldots,A_r$ are pairwise disjoint
Borel sets with $A_1\cup\cdots\cup A_r=S$, and let $D>0$ be an
upper bound for the diameter of all the $A_i$. Choose some $a_i\in
A_i$ for all $i$.

Approximate the characteristic function $\mathds{1}_{A_i}$ by a
polynomial $f_i\geq 0$ on $S$, such that $ \|\m(A_i) -\int f_i
d\m||_{\infty} \leq \frac{D}{r}.$ This can be done by Urysohn's
Lemma, the Stone-Weierstra\ss{} Theorem and  the Theorem of
Majorized Convergence for $\m$.  Then consider the operator
$$\Ph_{\underline{f},\underline{a}}\colon p\mapsto\sum_{i=1}^r p(a_i)\cdot \Ph(f_i),$$ which belongs to
$\W$. For any polynomial $p$ and any $x\in S$ we have
\begin{align*}|\Ph(p)(x)-
\Ph_{\underline{f},\underline{a}}(p)(x)|&=\left|\int pd\mu_x
-\sum_ip(a_i)\int f_id\mu_x \right| \\ &\leq \left|\int pd\mu_x
-\int\sum_i p(a_i)\mathds{1}_{A_i} d\mu_x\right| \\ &\quad +\left|
\int\sum_i p(a_i)\mathds{1}_{A_i} d\mu_x
-\int\sum_ip(a_i)f_id\mu_x \right|
\\ &\leq \mu_x(S)\cdot ||p-\sum_i
p(a_i)\mathds{1}_{A_i}||_{\infty}
\\&\quad +||p||_{\infty}\cdot\sum_i\left|\int\mathds{1}_{A_i}d\mu_x-\int
f_id\mu_x\right| \\&\leq ||\Ph(1)||_{\infty}\cdot||p-\sum_i
p(a_i)\mathds{1}_{A_i}||_{\infty} \\&\quad +
||p||_{\infty}\cdot\sum_i\left|\m(A_i)(x) -\left(\int
f_id\m\right)(x)\right| \\&\leq ||\Ph(1)||_{\infty}\cdot||p-\sum_i
p(a_i)\mathds{1}_{A_i}||_{\infty} +D\cdot ||p||_{\infty}.
\end{align*}
By the mean value theorem applied to $p$ we obtain $$||p-\sum_i
p(a_i)\mathds{1}_{A_i}||_{\infty}\leq D\cdot ||J(p)||_{\infty},$$
where $J(p):=1+ \sum_{j=1}^n\left(\frac{\partial}{\partial
X_j}p\right)^2.$ So we have shown $$||\Ph(p)-
\Ph_{\underline{f},\underline{a}}(p)||_{\infty}\leq
D\cdot\left(||\Ph(1)||_{\infty}\cdot ||J(p)||_{\infty} +
||p||_{\infty}\right)$$ for all $p\in\R[\X]$.

So if a sequence of partitionings of $S$ is chosen such that the
diameter bound $D$ gets arbitrary small (which can obviously be
done since $S$ is compact), then the corresponding sequence of
operators from $\W$ converges to $\Ph$, pointwise on $\R[\X]$.
\end{proof}

\section{Some open problems}

We include a collection of open problems:

\begin{itemize}
\item[(1)] For a given measure $\m$ on $\B(S)$ with values in $
\F(S)$ or $\mathbb{B}(S)$, find criteria for the operator defined
by $\m$ as in Theorem \ref{weak} and Theorem \ref{strong} to map
polynomials to polynomials.

\item[(2)] Find $S$-nonnegativity preservers that are compact but
do not have a finite dimensional range. For compact $S$ this means
to  find an $S$-nonnegativity preserver $\Ph$ such that the family
$(\Ph^*_A)_{A\in\B(S)}$ is equicontinuous, but not polynomial of
bounded degree. Could it be true that compact implies finite
dimensional range for nonnegativity preservers?

\item[(3)] The same question as in (2), but with weakly compact
operators that are not compact. For compact $S$, find an
nonnegativity preserver such that all $\Ph^*_A$ are continuous,
but these functions do not form an equicontinuous family.

\item[(4)]\label{fin} Is every $S$-nonnegativity preserver with
finite dimensional range of the form $$\Ph\colon p\mapsto
\sum_{i=1}^rf_i\cdot\int_S pd\nu_i$$ with nonnegative polynomials
$f_i$ and \textit{nonnegative} Borel measures $\nu_i$ on $S$? In
view of Haviland's Theorem, can the representation
$$\Ph=\Ph_{\underline{f},\underline{L}}$$ as given in the proof of
Theorem \ref{fira} be chosen such that all $f_i\geq 0$ on $S$ and
all $L_i$ map $S$-nonnegative polynomials to nonnegative reals?
The following example might be interesting in this regard: Let
$\la$ denote the Lebesque measure. Consider the following operator
on $\R[t]$:
\begin{align*}\Ph\colon p \mapsto
(t+2)\cdot\int_{-1}^1pd\la \quad - t^2\cdot\int_0^1pd\la.
\end{align*} $\Ph$ is $[-1,1]$-nonnegativity preserving, the
polynomials $f_1=t+2$ and $f_2=t^2$ are nonnegative on $[-1,1]$,
but the linear functional $L_2\colon p\mapsto -\int_0^1p d\la$ is
not integration with respect to a nonnegative measure on $[-1,1]$.
However,
$$\Ph\colon p\mapsto (t+2)\cdot\int_{-1}^0 pd\la +
(t+2-t^2)\cdot\int_0^1pd\la$$ is a representation of $\Ph$ as
desired.

\item[(5)] In case $S$ is compact, which nonnegativity preservers
can be approximated by elements from $\W$ with respect to the
\textit{operator norm} instead of the strong operator topology as
in Theorem \ref{approx}? As all elements from $\W$ are compact
operators on $C(S)$, it is a well known fact that only compact
operators can be approximated like that. Can every compact
$S$-nonnegativity preserver be approximated?
\end{itemize}

\end{document}